\documentclass[10pt]{amsart}
\usepackage{amsmath,amssymb,latexsym}
\usepackage[T1]{fontenc}
\usepackage{enumerate}

\def\Ind#1#2{#1\setbox0=\hbox{$#1x$}\kern\wd0\hbox to 0pt{\hss$#1\mid$\hss}
\lower.9\ht0\hbox to 0pt{\hss$#1\smile$\hss}\kern\wd0}
\def\ind{\mathop{\mathpalette\Ind{}}}
\def\Notind#1#2{#1\setbox0=\hbox{$#1x$}\kern\wd0\hbox to 0pt{\mathchardef
\nn="3236\hss$#1\nn$\kern1.4\wd0\hss}\hbox to 0pt{\hss$#1\mid$\hss}\lower.9\ht0
\hbox to 0pt{\hss$#1\smile$\hss}\kern\wd0}
\def\nind{\mathop{\mathpalette\Notind{}}}

\def\cl{\mathrm{cl}}
\def\dcl{\mathrm{dcl}}
\def\acl{\mathrm{acl}}
\def\cb{\mathrm{cb}}

\def\P{\mathbb P}
\def\M{\mathfrak M}

\def\wU{\mathrm{U_\omega}}
\def\U{\mathrm{U}}

\def\Aut{\mathrm{Aut}}
\def\tp{\mathrm{tp}}
\def\k{k}
\def\stp{\mathrm{stp}}
\def\wind{\ind^\omega}
\def\nwind{\nind^\omega}

\theoremstyle{plain}
\newtheorem{theorem}{Theorem}[section]
\newtheorem{prop}[theorem]{Proposition}
\newtheorem{fact}[theorem]{Fact}
\newtheorem{lemma}[theorem]{Lemma}
\newtheorem{cor}[theorem]{Corollary}

\theoremstyle{definition}
\newtheorem{defn}[theorem]{Definition}
\newtheorem{remark}[theorem]{Remark}

\newtheorem{expl}[theorem]{Example}

\newtheorem{quest}{Question}

\def\pf{\par\noindent{\em Proof. }}

\title{On the class of flat stable theories}
\author{Daniel Palac\'in and Saharon Shelah}
\thanks{This paper corresponds to 1133 in Shelah's publication list. \newline Both authors were partially supported by the European Research Council grant 338821. The first author was also partially supported by the project MTM2014-59178-P}

\address{Einstein Institute of Mathematics, The Hebrew University of Jerusalem, Givat Ram 9190401, Jerusalem, Israel}
\email{daniel.palacin@mail.huji.ac.il}

\address{Einstein Institute of Mathematics, The Hebrew University of Jerusalem, Givat Ram 9190401, Jerusalem, Israel \newline \indent Department of Mathematics, Hill Center - Busch Campus,
Rutgers, The State University of New Jersey, 110 Frelinghuysen Road, Piscataway, NJ 08854-8019
USA} 
\email{shelah@math.huji.ac.il}

\begin{document}

\begin{abstract} 
A new notion of independence relation is given and associated to it, the class of flat theories, a subclass of strong stable theories including the superstable ones is introduced. More precisely, after introducing this independence relation, flat theories are defined as an appropriate version of superstability. It is shown that in a flat theory every type has finite weight and therefore flat theories are strong. Furthermore, it is shown that under reasonable conditions any type is non-orthogonal to a regular one. Concerning groups in flat theories, it is shown that type-definable groups behave like superstable ones, since they satisfy the same chain condition on definable subgroups and also admit a normal series of definable subgroup with semi-regular quotients.
\end{abstract}

\maketitle

\section{Introduction}

The notions of forking, orthogonality and regular types, among others, play a fundamental role in understanding the structure of stable theories. These were not only essential to carry out the classification programme, inside stable theories, but also have turned out to be relevant for the developments of geometric stability theory.

A stationary type is {\em regular} if it is orthogonal to all its forking extensions; recall that two stationary types $p$ and $q$ are {\em orthogonal} if, for any set $C$ over which both types are based and any realizations $a\models p|C$ and $b\models q|C$, we have that $a\ind_C b$. Minimal types are the simplest example of regular types, where forking means being algebraic. Similar to minimal ones, regular types carry a notion of geometry associated to their set of realizations, and hence a dimension. 
Their main feature is that any type can be coordinatizated by regular ones, as long as the theory contains enough regular types. Consequently, their associated geometries determine many properties of the theory. 

Formally, the fact that a theory has enough regular types can be rephrased as follows: Every type is non-orthogonal to a regular one. This holds for superstable theories but this property is not exclusive of superstability. Therefore, one may try to find reasonable conditions beyond superstability which yield the existence of enough regular types. In this paper we pursue this line of investigation on an attempt to find some reasonable structure theory beyond superstability.  

We introduce the class of flat theories, a subclass of stable theories which extends superstability, and analyse the existence of regular types in this context. More precisely, in Section 2 we define the notion of $\omega$-forking\footnote{Originally, called gorking by the second author.}, which implies the usual notion of forking, and show that in a stable theory it satisfies the usual properties of independence (see Theorem \ref{T:Main1}), except algebraicity since it can be the trivial relation.  Afterwards, a flat theory is defined as a stable theory where every type does not $\omega$-fork over a finite set. Since non-forking implies non-$\omega$-forking, it follows immediately that a superstable theory is flat. As in the superstable case, a notion of ordinal-valued rank among types, called $\wU$-rank, is available and we point out some of its basic properties, such as the Lascar inequalities. 

In the third section, a more careful analysis of flat theories is carried out. Roughly speaking, we see that any type has a non-$\omega$-forking extension which is non-orthogonal to a regular type. Consequently, every type is close to be non-orthogonal to a regular one, see Theorem \ref{T:Main2}. In particular, if all forking extensions of a type are also $\omega$-forking, then it is non-orthogonal to a regular type. This is Corollary \ref{C:Reg}. Nevertheless, we cannot ensure that in general every type is non-orthogonal to a regular one, but we show that flat theories are strong (Theorem \ref{T:Strong}) and consequently every type is non-orthogonal to a type of weight one. In fact, this holds locally for a flat type under the mere assumption that the theory is stable.

Finally, in the last section groups in flat theories are analysed. We show that any type-definable group in a flat theory looks like a superstable one, in the sense that they satisfy the same descending chain condition on definable subgroups and also admit a semi-regular decomposition. It should be noted that, while the notion of $p$-semi-regularity (also $p$-simplicity) originated in \cite[Chapter V]{SheClas}, here semi-regularity corresponds to a reformulation due to Hrushovski. Hence, in Theorem \ref{T:Group}, by a semi-regular decomposition we mean that every such flat group admits a finite series of normal subgroups such that any generic type of each quotient is domination-equivalent to suitable finite product of some regular type.

\section{A new independence relation}

From now on, we work inside the monster model of a complete stable first-order theory, and we assume that the reader is familiarized with the general theory of stability theory.

\subsection{Skew dividing and $\omega$-forking} We introduce the notion of skew $k$-dividing and $\k$-forking for a natural number $k \ge 1$. 

\begin{defn}
Let $\pi(\bar x)$ be a partial type. It is said to 
skew $k$-divide over $A$ if there is an $A$-indiscernible sequence $(\bar b_n)_{n<\omega}$ and a formula $\varphi(\bar x; \bar y_0,\ldots,\bar y_{\k-1})$ such that $$\pi(\bar x)\vdash \varphi(\bar x;\bar b_0,\bar b_2,\ldots,\bar b_{2(\k-1)}) \mbox{ and  }
\pi(\bar x)\vdash \neg\varphi(\bar x;\bar b_{i_0},\ldots,\bar b_{i_{\k-1}})$$ for any  $i_0<\ldots<i_{\k-1}<2\k$ with $(i_0,i_1,\ldots,i_{\k-1})\neq (0,2,\ldots,2(\k-1)).$
\end{defn}

In fact, in the definition of skew dividing we may allow formulas with parameters.

\begin{remark}\label{RemDivideParam}
A partial type $\pi(\bar x)$ skew $k$-divides over $A$ if and only if there are a formula $\varphi(\bar x; \bar y_0,\ldots,\bar y_{\k-1},\bar z)$, a tuple $\bar c$, and an $A\bar c$-indiscernible sequence $(\bar b_n)_{n<\omega}$  such that $$\pi(\bar x)\vdash \varphi(\bar x;\bar b_0,\bar b_2,\ldots,\bar b_{2(\k-1)},\bar c) \ \mbox{ and  } \
\pi(\bar x)\vdash \neg\varphi(\bar x;\bar b_{i_0},\ldots,\bar b_{i_{\k-1}},\bar c)$$ for any  $i_0<\ldots<i_{\k-1}<2\k$ with $(i_0,i_1,\ldots,i_{\k-1})\neq (0,2,\ldots,2(\k-1)).$
\end{remark}
\pf Left to right is obvious by the definition of skew $k$-dividing. To prove the other direction, assume that the condition holds for $\varphi(\bar x; \bar y_0,\ldots,\bar y_{\k-1},\bar z)$, a tuple $\bar c$, and a sequence $(\bar b_n)_{n<\omega}$. Set $\bar y_i'=\bar y_i\bar z$ and $\bar b'_n=\bar b_n\bar c$. Then the formula $\psi(\bar x;\bar y_0',\ldots,\bar y_{\k-1}')$ defined as $\varphi(\bar x;\bar y_0,\ldots,\bar y_{\k-1},\bar z)$ and the sequence $(\bar b_n')_{n<\omega}$ witness that $\pi(\bar x)$ skew $k$-divides over $A$. \qed

\begin{defn}
A partial type $\pi(\bar x)$ is said to $k$-fork over $A$ if it implies a finite disjunction of formulas, each of them skew $k$-dividing over $A$. 
\end{defn}

In other words, the set of formulas that $k$-fork over $A$ is nothing else than the ideal generated by the formulas that skew $k$-divide over $A$. Furthermore, note that both notions are preserved under automorphisms of the ambient model. 

\begin{remark}\label{R:1} The following holds:
\begin{enumerate} 
 \item If $\pi_1(\bar x)\vdash \pi_2(\bar x)$ and $\pi_2(\bar x)$ skew $k_2$-divides over $A_2$, then $\pi_1(\bar x)$ skew $k_1$-divides over $A_1$ for any $k_1\le k_2$ and $A_1\subseteq A_2$. The same holds for $k$-forking. 
 \item If $\pi(\bar x)$ skew $k$-divides over $A$, then so does some finite subset $\pi_0(\bar x)$ of $\pi(\bar x)$. Similarly for $k$-forking. 
  \item If a partial type $\pi(\bar x)$ skew $k$-divides over $A$, then so does it over $\acl(A)$. 
 \item Extension property. If a partial type $\pi(\bar x)$ over $B$ does not $\k$-fork over $A$, then there is $p(\bar x)\in S(B)$ extending $\pi(\bar x)$ which does not $\k$-fork over $A$.
\end{enumerate}
\end{remark}
\pf We only prove (1) for skew dividing, the rest is standard. Assume that $\varphi=\varphi(\bar x;\bar y_0,\ldots,\bar y_{\k_2-1})$ and $(\bar b_\alpha)_{\alpha<\omega}$ witness that $\pi_2(\bar x)$ skew $k_2$-divides over $A_2$. Now, set $\bar z=\bar y_{\k_1}\ldots \bar y_{\k_2-1}$ and let $\psi(\bar x;\bar y_0,\ldots,\bar y_{\k_1-1},\bar z)=\varphi$.  Then the result follows from Remark \ref{RemDivideParam} by enlarging the sequence $(\bar b_\alpha)_{\alpha<\omega}$ to $(\bar b_\alpha)_{\alpha<\omega+\omega}$ and taking $\bar c=\bar b_\omega\ldots \bar b_{\omega+\k_2-\k_1+1}$. \qed

\begin{lemma}\label{L:k=1}
A partial type $\pi(\bar x)$ does not fork over $A$ if and only if it does not $1$-fork over $A$.
\end{lemma}
\pf It is clear that a global type is Lascar invariant over $A$ if and only if it does not $1$-fork over $A$. Thus, the statement follows as non-forking and non-$1$-forking satisfy the extension property. \qed

Nevertheless, for $k>1$ forking and $k$-forking does not agree in general.

\begin{expl}
Consider the first-order theory of an infinite set and let $\phi(x;y)$ be the formula $x=y$. For any  element $a$, we have that the partial type $\{\phi(x;a)\}$ forks over $\emptyset$, but it does not $2$-fork. 
\end{expl}

\begin{lemma}\label{L:SD-Indisc}
If the type $\tp(\bar a/B)$ does not skew $k$-divide over $A$, then for any $A$-indiscernible sequence $I$ contained in $B$, there is some $J\subseteq I$ with $|J|<\k$ such that $I\setminus J$ is an indiscernible set over $AJ\bar a$.
\end{lemma}
\pf 
Inductively on $n \le k$, we obtain a strictly increasing sequence of natural numbers $(k_n)_{n\le k}$ with $k_0=0$ for which there is a subsequence $J_n=(\bar b_m)_{m\in (k_n,k_{n+1})}$ of $I$ without repetitions and a formula $\phi_n(\bar x; \bar y_0,\ldots, \bar y_{k_{n+1}-1},\bar z)$ such that:
\begin{itemize}
\item[$\bullet_1$] there is some finite tuple $\bar c_n$ in $AI$ such that $\phi_n(\bar a ;\bar b_0,\ldots, \bar b_{k_{n+1}-1},\bar c_n)$ holds, 
\item[$\bullet_2$] there is some finite subset $I_n$ of $I$, containing $I_{n-1}$, such that the relation $\neg\phi_n(\bar a;\bar b_0,\ldots,\bar b_{k_n -1},\bar y_{k_n}',\ldots, \bar y_{k_{n+1}-1}',\bar c_n)$ also holds for any $\bar b_{k_n}',\ldots , \bar b_{k_{n+1}-1}'$ in $I\setminus I_n$, and
\item[$\bullet_3$] $k_{n+1}$ is minimal with these properties.

\end{itemize}
Let $\Delta_n$ be the closure of $\phi_n$ under permuting the variables, and let $\Delta$ be the union of all these $\Delta_n$. 
As $\tp(\bar a/AI)$ does not $k$-fork over $A$, there is some $n_* \le k$ for which we cannot keep doing the construction for $n_*$. If $|J_{ < n_*}|<k$, then as the truth value of any formula over $A\bar a J_{<n_*}$ is constant in a cofinal segment of $I\setminus J_{< n_*}$, the choice of $n_*$ yields that $I\setminus J_{< n_*}$ is indiscernible over $A\bar a J_{< n_*}$, as desired. 

Assume now that $|J_{< n_*}| \ge k$. Thus, by construction we have that
\begin{itemize}
\item[$\otimes_1$] $k_{n_*} \ge k$ and so $n_* \ge 1$. 
\end{itemize}
Now, take some $\bar b_0^*,\ldots,\bar b_{2k_{n_*} -1}^* \in I \setminus (I_{< n_*} \cup J_{< n_*})$ without repetitions and set $\bar b_* =\bar b_{k_{n_*}}^* \ldots \bar b_{2k_{n_*} -1}^*$ and $\bar b_i^1 = \bar b_i$ and $\bar b_i^2 = \bar b_{i}^*$ for $i<k_{n_*}$. 

Put $\bar z=\bar z_0\ldots \bar z_{n_*}$ and also $\bar c_* = \bar c_0\ldots \bar c_{n_*}$. Then let 
$\psi(\bar x;\bar y_0,\ldots,\bar y_{k_{n_*+1}-1},\bar z,\bar b_*)$ denote the conjunction of the finite partial type $\tp_\Delta(\bar a \bar b_0 \ldots \bar b_{k_{n_*}-1} \bar c_* /\bar b_*)$.

Notice that $\psi(\bar a;\bar b_0^1,\ldots,\bar b_{k_{n_*}-1}^1,\bar c_*,\bar b_*)$ holds by construction. Thus, the set $\Lambda$ of functions $\eta:\{0,\ldots,k_{n_*}-1\} \longrightarrow \{ 1 , 2 \}$ such that $\psi(\bar a;\bar b_0^{\eta(0)},\ldots,\bar b _{k_{n_*}-1}^{\eta(k_{n_*}-1)},\bar c_*,\bar b_*)$ holds is non-empty. Moreover, note that $\bar b_0^1,\bar b_0^2, \ldots , \bar b_{k_{n_*}-1}^1,\bar b_{k_{n_*}-1}^2$ cannot witness that $\tp(\bar a/AI)$ $k$-forks over $A\bar b_*$, as $\tp(\bar a/AI)$ does not $k$-fork over $A$. Thus, there is some $\eta \in \Lambda$ such that $u_\eta = \{ l < k_{n_*} : \eta(l) = 2 \}$ is non-empty. 

Fix some $\eta\in \Lambda$ with $u_\eta \neq \emptyset$, and let $n_{**}<n_*$ be minimal with the property that $u_\eta \cap [k_{n_{**}},k_{n_{**}+1}) \neq \emptyset$. After permuting the variables if necessary we may assume that $u_\eta \cap [k_{n_{**}},k_{n_{**}+1}) = [k_*, k_{n_{**}+1})$ for some $k_* \in [k_{n_{**}},k_{n_{**}+1})$.
Now, set $\bar d_0 = \bar b_0 \ldots \bar b_{k_{n_{**}}-1}$,  $\bar d_1 = \bar b_{k_{n_{**}}}\ldots \bar b_{k_{*}-1}$, $\bar d_2^1 = \bar b_{k_{*}}^1 \ldots \bar b_{k_{n_{**}+1}-1}^1$ and $\bar d_2^2 = \bar b_{k_{*}}^{\eta(0)} \ldots \bar b_{k_{n_{**}+1}-1}^{\eta(k_{n_{**}+1}-1)}$. Hence, by $\bullet_1$ we have that
\begin{itemize}
\item[$\otimes_2$] $\phi_{n_{**}}(\bar a;\bar d_0,\bar d_1,\bar d_2^1,\bar c_n)$ holds.
\end{itemize}
On the other hand, the choice of $\psi$ and $\eta\in\Lambda$ yield that
\begin{itemize}
\item[$\otimes_3$] $\phi_{n_{**}}(\bar a;\bar d_0,\bar d_1,\bar d_2^1,\bar c_n)$ holds if and only if so does $\phi_{n_{**}}(\bar a;\bar d_0,\bar d_1,\bar d_2^2,\bar c_n)$.
\end{itemize}
Hence, by $\otimes_2$ and $\otimes_3$ we get that
\begin{itemize}
\item[$\otimes_4$] $\phi_{n_{**}}(\bar a;\bar d_0,\bar d_1,\bar d_2^2,\bar c_n)$ holds.
\end{itemize}
Observe that since $n_{**}<n_*$, we have that $I_{n_{**}}$ is contained in $I_{<n_{*}}$ and so $\bar d_2^2$ is formed with elements from $I\setminus I_{n_{**}}$. Thus, by $\bullet_2$  we have that
\begin{itemize}
 \item[$\otimes_5$] for any pairwise distinct elements $\bar b_{k_{n_{**}}}',\ldots, \bar b_{k_{*}-1}'$ of $I\setminus I_{n_{**}}$ we have that $\neg \phi_{n_{**}}(\bar a;\bar d_0,\bar b_{k_{n_{**}}}',\ldots, \bar b_{k_{*}-1}',\bar d_2^2,\bar c_{n_{**}})$ holds.
\end{itemize}
Therefore, by $\otimes_4$, $\otimes_5$ and setting $\bar c_{n_{**}}' = \bar d_2^2\bar c_{n_{**}}$, we contradict the minimality of $k_{n_{**}+1}$ given by $\bullet_3$ since $k_*<k_{n_{**}+1}$. This finishes the proof. \qed

\begin{prop}\label{P:Equiv1} Let $\bar a$ be a finite tuple, and let $A$ be a subset of an $(|A|+|T|)^+$-saturated model $M$. Then, the following are equivalent:
\begin{enumerate}
 \item The type $\tp(\bar a/M)$ does not $k$-fork over $A$.
 \item For any $A$-indiscernible sequence $I$ contained in $M$, there is some $J\subseteq I$ with $|J|<\k$ such that $I\setminus J$ is an indiscernible set over $AJ\bar a$.
\end{enumerate}
Moreover, the above properties implies the following:
\begin{enumerate}
\item[(3)] For any $A$-independent sequence $I$ contained in $M$, there is some $J\subseteq I$ with $|J|<\k$ such that $I\setminus J$ is independent from $AJ\bar a$ over $A$.
\end{enumerate}
\end{prop}
\pf $(1)\Rightarrow (2)$ is the lemma above. To show $(2)\Rightarrow (1)$, suppose that $\tp(\bar a/M)$ $k$-forks over $A$. Thus there is some formula $\psi(\bar x)\in \tp(\bar a/M)$ that $k$-forks over $A$. That is, the formula $\psi(\bar x)$ implies a finite disjunction of formulas that skew $k$-divide over $A$. Note that by saturation of $M$ each of these formulas can be taken with parameters over $M$. Thus we can find a formula $\phi(\bar x;\bar y_0,\ldots,\bar y_{k-1})$ and an $A$-indiscernible sequence $(\bar b_n)_{n<\omega}$ witnessing this. Notice again by saturation that we may take $(\bar b_n)_{n<\omega}$ inside $M$. By (2), there is a finite subset $J$ of $\omega$ with $|J|<k$ such that $I\setminus J$ is indiscernible over $A\bar aJ$. Thus, there is some $l<k$ such that $2l\not\in J$ and so $\phi(\bar a;\bar b_{i_0},\ldots,\bar b_{i_{k-1}})$ holds by indiscernibility  taking $i_j = 2j$ for $j\neq l$ and $i_l = 2l+1$, a contradiction.

Finally we see that $(2)\Rightarrow (3)$. Assume that $I=(\bar a_s)_{s<\alpha}$ is an $A$-independent sequence and consider a Morley sequence $(\bar a_{s,t})_{t<\alpha}$ in $p_s=\stp(\bar a_s/A)$ with $\bar a_s = \bar a_{s,s}$ in a way that the array $(\bar a_{s,t})_{s,t <\alpha}$ is an independent set over $A$. By saturation, we may take this array inside $M$. Set $\bar b_t = (\bar a_{s,t})_{s<\alpha}$ and note that it realizes the stationary type $\bigotimes_{s<\alpha} p_s$. Consequently, as $(\bar b_t)_{t<\alpha}$ is $A$-independent, we obtain that it is an $A$-indiscernible sequence. Hence, by (2) there exists some $J\subseteq \alpha$ with $|J|<k$ such that $(\bar b_t)_{t\not\in J}$ is indiscernible over $A\bar a \cup\{\bar b_t \}_{t\in J}$. Whence, since  $(\bar b_t)_{t\not\in J}$ is Morley sequence in $\bigotimes_{s<\alpha} p_s$, we have that $\bar a \cup\{\bar b_t \}_{t\in J}$ is independent from $(\bar b_t)_{t\not\in J}$ over $A$, and so $I\setminus \{\bar a_s\}_{s\in J}$ is independent from $A\bar a\cup\{\bar a_t\}_{t\in J}$, as desired.  \qed

\begin{remark}
In view of Remark \ref{R:1}(4) and Proposition \ref{P:Equiv1} we could have defined $\k$-forking as follows: A partial type $\pi(\bar x)$ does not $k$-fork over $A$ if it can be extended to a complete type $p(\bar x)$ over an $(|A|+|T|^+)$-saturated model $M$ such that for any $\bar a\models p$ and any $A$-indiscernible sequence $I$ contained in $M$, there is some $J\subseteq I$ with $|J|<\k$ such that $I\setminus J$ is an indiscernible set over $AJ\bar a$.
\end{remark}

\begin{lemma}\label{L:Trans}
If $\tp(\bar a_1/B)$ does not $\k_1$-fork over $A\subseteq B$ and $\tp(\bar a_2/B\bar a_1)$ does not $\k_2$-fork over $A\bar a_1$, then $\tp(\bar a_1\bar a_2/B)$ does not $(\k_1+\k_2)$-fork over $A$. 
\end{lemma}
\pf Consider an $(|A|+|T|)^+$-saturated model $M$ extending $B$. By extension, {\it i.e.} Remark \ref{R:1}(4), there is some $\bar a_1'\models\tp(\bar a_1/B)$ such that $\tp(\bar a_1'/M)$ does not $\k_1$-fork over $A$. Let $\bar a_2'$ be such that $\bar a_1\bar a_2\equiv_B \bar a_1'\bar a_2'$ and note that $\tp(\bar a_2'/B,\bar a_1')$ does not $\k_2$-fork over $A\bar a_1'$ by invariance. Hence, again by extension there is some $\bar a_2''\equiv_{B\bar a_1'} \bar a_2'$ such that $\tp(\bar a_2''/M,\bar a_1')$ does not $\k_2$-fork over $A\bar a_1'$. Now, given an $A$-indiscernible sequence $I$ contained in $M$, applying twice Lemma \ref{L:SD-Indisc} we find two disjoint subsets $J_1$ and $J_2$ of $I$ with $|J_1|<k_1$ and $|J_2|<k_2$ such that $I\setminus (J_1\cup J_2)$ is indiscernible over $AJ_1J_2\bar a_1'\bar a_2''$. Hence by Proposition \ref{P:Equiv1}, we get that the type $\tp(\bar a_1'\bar a_2''/M)$ does not $(k_1+k_2)$-fork over $A$ and neither does $\tp(\bar a_1\bar a_2/M)$ by invariance. \qed 

In the light of the previous result we introduce the following notion.

\begin{defn}
A partial type $\pi(\bar x)$ $\omega$-forks over $A$ if it $\k$-forks over $A$ for every natural number $\k$. We write $\bar a\wind_A B$ whenever $\tp(\bar a/AB)$ does not $\omega$-fork over $A$. 
\end{defn}
This notion satisfies the usual axioms of a ternary independence relation. 

\begin{theorem}\label{T:Main1} The ternary relation $\wind$ defined among imaginary sets satisfies the following properties:
\begin{enumerate}
 \item Invariance: $\wind$ is invariant under $\Aut(\M)$. 
 \item Finite character: $\bar a\wind_A B$ if and only if $\bar a'\wind_A B'$ for any finite tuple $\bar a'\subseteq \bar a$ and any finite set $B'\subseteq B$. 
 \item Transitivity: If $\bar a\wind_{A\bar b} B$ and $\bar b\wind_A B$, then $\bar a\bar b\wind_A B$.
 \item Base monotonicity: If $\bar a\wind_A BC$, then $\bar a\wind_{AB} C$. 
 \item Extension: If $\bar a\wind_A B$, then for any $C$ there exists some $\bar a'\equiv_{AB} \bar a$ with $\bar a'\wind_A BC$.
 \item Local character: For every finite tuple $\bar a$ and any set $B$ there is some $A\subseteq B$ with $|A|<|T|^+$ such that $\bar a\wind_A B$. 
 \item Symmetry: $\bar a\wind_A \bar b$ if and only if $\bar b\wind_A \bar a$. 
 \end{enumerate} 
\end{theorem}
\pf Invariance, finite character and base monotonicity are straightforward from the definition. Extension follows from Remark \ref{R:1}(4), and transitivity from Lemma \ref{L:Trans}. Furthermore, notice that the relation $\wind$ satisfies local character by stability, Lemma \ref{L:k=1} and Remark \ref{R:1}(1).

Finally, symmetry holds by \cite[Theorem 2.5]{Adler1}. We offer a shorter proof using stability.  By extension and finite character we can find an indiscernible  sequence $(\bar a_i)_{i<|T|^+}$ in $\tp(\bar a/A,\bar b)$ such that $\bar a_i\wind_A \bar b,(\bar a_j)_{j<i}$ for every $i<|T|^+$. In particular, we have that $\bar a_i\wind_A (\bar a_j)_{j<i}$. As any indiscernible sequence is an indiscernible set, we obtain inductively on $i$ that $(\bar a_j)_{j<i}\wind_A \bar a_i$ by invariance, finite character and transitivity. Now, local character of $\wind$ implies the existence of some $i<|T|^+$ such that $\bar b\wind_{A,(\bar a_j)_{j<i}} \bar a_i$. Hence, we obtain that $\bar b,(\bar a_j)_{j<i}\wind_A \bar a_i$ by transitivity and so $\bar b\wind_A \bar a_i$ by finite character. Whence, we obtain the result by invariance. \qed

\subsection{Flat theories} Next we introduce a subclass of stable theories which includes the superstable ones. 
\begin{defn}
A stable theory is {\em flat} if for every finite tuple $a$ and every set $A$, there exists a finite subset $A_0$ of $A$ such that $a \wind_{A_0} A$.
\end{defn}

It follows from the definition of flatness and $\omega$-forking that any superstable theory is flat. Nevertheless, not every flat theory is superstable. The following exhibit can be seen as the archetypical example of flat non superstable theory.

\begin{expl}
Consider the first-order theory of countably many nested equivalence relations $\{E_i(x,y)\}_{i<\omega}$ such that $E_0(x,y)$ has infinitely many classes, and each $E_i$-class can be partitioned into infinitely many $E_{i+1}$-classes. This is a stable flat theory which is not superstable theory.
\end{expl}

The importance of flatness is that the foundation rank associated to the binary relation  of being an $\omega$-forking extension among finitary complete types over sets takes ordinal values. 

\begin{defn}
The $\wU$-rank is the least function from the collection of all types (with parameters from the monster model) to the set of ordinals or $\infty$ satisfying for every ordinal $\alpha$:
\begin{center}
$\wU(p) \ge \alpha+1$ if there is an $\omega$-forking extension $q$ of $p$ with $\wU(q)\ge \alpha$.
\end{center}
As usual, to easer notation we write $\wU(a/A)$ for $\wU(\tp(a/A))$.
\end{defn}
The $\wU$-rank is invariant under automorphism and clearly $\wU(p)\le \U(p)$ for any finitary complete type $p$. Since every type does not fork over a set of cardinality at most $|T|$, there are at most $2^{|T|}$ different $\U$-ranks and so at most $2^{|T|}$ different $\wU$-ranks. As these values form an initial segment of the ordinals, all of them  are smaller than $(2^{|T|})^+$. Thus, it follows that every type of $\wU$-rank $\infty$ has a forking extension of $\wU$-rank $\infty$. 

\begin{prop} The following holds:
\begin{enumerate}
\item If $q$ extends $p$, then $\wU(p)\ge \wU(q)$. Moreover, if $q$ is a non-$\omega$-forking extension of $p$, then $\wU(p)=\wU(q)$.
\item A theory is flat if and only if $\wU(p)<\infty$ for every finitary complete (real) type $p$.
\end{enumerate}
\end{prop}
\begin{proof}
The proof is standard and it is left to the reader.
\end{proof}

\begin{remark}It follows from the definition of $\wU$-rank that a finitary complete type has $\wU$-rank zero if and only if it has no $\omega$-forking extensions. In particular, by the extension property we have that $\wU(a/A)= 0 $ if and only if $a\wind_A a$. 
\end{remark}

Recall that every ordinal $\alpha$ can be written in the Cantor normal form as a finite sum $\omega^{\alpha_1}\cdot n_1 + \ldots + \omega^{\alpha_k}\cdot n_k$ for ordinals $\alpha_1>\ldots>\alpha_k$ and natural numbers $n_1,\ldots,n_k$.  If additionally $\beta=\omega^{\alpha_1}\cdot m_1 + \ldots + \omega^{\alpha_k}\cdot m_k$, then the sum $\alpha\oplus\beta$, which is defined as $\omega^{\alpha_1}\cdot (n_1+m_1) + \ldots + \omega^{\alpha_k}\cdot (n_k+m_k)$, is commutative. In fact, the sum $\oplus$ is the smallest symmetric strictly increasing function $f$ among pairs of ordinals such that $f(\alpha,\beta+1)=f(\alpha,\beta)+1$. 

The proof of the following result is standard, see for instance \cite[Theorem 4]{CasWag}.
\begin{theorem}[Lascar Inequalities] The following holds:
\begin{enumerate}
\item $\wU(a/Ab) + \wU(b/A) \le \wU(ab/A) \le \wU(a/Ab) \oplus \wU(b/A)$.
\item If $\wU(a/Ab)<\infty$ and $\wU(a/A) \ge \wU(a/Ab) \oplus \alpha$ for some ordinal $\alpha$, then $\wU(b/A) \ge \wU(b/Aa) \oplus \alpha$.
\item If  $\wU(a/Ab)<\infty$ and $\wU(a/A) \ge \wU(a/Ab) + \omega^\alpha$ for some ordinal $\alpha$, then $\wU(b/A) \ge \wU(b/Aa) + \omega^\alpha$.
\item If $a\wind_A b$, then $\wU(ab/A) = \wU(a/A) \oplus \wU(b/A)$.
\end{enumerate}
\end{theorem}

We finish this section by pointing out the existence of a link between forking and $\omega$-forking via canonical bases and types of $\wU$-rank zero. 

\begin{prop}\label{P:CanBase}
If $a\wind_A b$, then $\wU(\cb(\stp(a/Ab))/A) = 0$. Furthermore, the opposite holds assuming that $\wU(a/A)<\infty$.
\end{prop}
\begin{proof}
As non-$\omega$-forking independence has finite character, notice that the type $\tp(\cb(\stp(a/Ab))/A)$ has $\wU$-rank zero if and only if $\wU(c/A)=0$ for any finite tuple of $\cb(\stp(a/Ab))$. Now, suppose that $a\wind_A b$ and let $\bar a= (a_i)_{i<\omega}$ be a Morley sequence in $\stp(a/Ab)$. Thus one can easily see that $\bar a\wind_A b$ using Theorem \ref{T:Main1}. Since any finite tuple $c$ of $\cb(\stp(a/Ab))$ belongs to $\dcl(\bar a)\cap \acl(Ab)$, we then have $c\wind_A c$ and so $\wU(c/A)=0$. 

For the opposite, assume that $\wU(a/A)<\infty$ and set $C=\cb(\stp(a/Ab))$. Thus, by the Lascar inequalities 
$$
\wU(a/A) \le \wU(a/A,C) \oplus \wU(C/A) = \wU(a/A,C),
$$
so $\wU(a/A) = \wU(a/A,C)<\infty$ and hence $a\wind_A C$. Moreover, since $a\ind_C Ab$ we have that $a\wind_{CA} b$ and therefore $a\wind_A b$ by transitivity, as desired.
\end{proof}

\section{Searching for enough regular types}

\subsection{Types without $\omega$-forking extensions} As we point out before, a type has $\wU$-rank zero if and only if it has no $\omega$-forking extensions. In this section we shall see that these types play a fundamental role towards the existence of enough regular types in flat theories.

Let $\P$ be an $\emptyset$-invariant family of partial types. A stationary type $p\in S(A)$ is {\em foreign} to $\P$ if for all sets $B\supseteq A$ and all realizations $a$ of $p|B$ we have that $a\ind_B c$ for any $c$ such that $\tp(c/B)$ extends some member of $\P$. The type $p$ is {\em (almost) $\P$-internal} if there exists some $B\supseteq A$, a realization  $a\models p|B$ and some tuple $\bar b=(b_1,\ldots, b_n)$ such that $\bar a\in\dcl(B,\bar b)$ ($\bar a\in\acl(B,\bar b)$, respectively) and each type $\tp(b_i/B)$ extends a member of $\P$. Finally, it is {\em $\P$-analysable} in $\alpha$ steps if for some realization $a$ of $p$ there is a sequence $(a_i)_{i<\alpha}$ in $\dcl(A,a)$ such that each type $\tp(a_i/A,(a_j)_{j<i})$ is $\P$-internal, and $a\in\acl(A,(a_i)_{i<\alpha})$. 

The following result, see \cite[Corollary 7.4.6]{Pil}, plays an essential role in this section.
\begin{fact}\label{F:Int}
If the type $\stp(a/A)$ is not foreign to $\P$, then there is some imaginary element $a_0\in \dcl(Aa) \setminus \acl(A)$ such that  $\stp(a_0/A)$ is $\P$-internal. 
\end{fact}

Let $\P_0$ denote the family of types of $\wU$-rank zero. It is easy to see that any finitary complete type which is $\P_0$-analysable in finitely many steps must have $\wU$-rank zero by the Lascar inequalities. Consequently, we obtain the following:

\begin{lemma}
If the type $\stp(a/A)$ is not foreign to $\P_0$, then there is some imaginary element $a_0\in \dcl(Aa) \setminus \acl(A)$ such that $\wU(a_0/A)=0$.
\end{lemma}

Given a set $A$, set $\cl_{\P_0}(A)$ to be the set of all elements $b$ such that $\tp(b/A)$ has $\wU$-rank zero. By \cite[Corollary 6]{Udi1} we obtain the following decomposition lemma, see also \cite[Corollary 6]{CasWag}. For the sake of completeness we give a (direct) proof. 
\begin{lemma}\label{L:Decomp}
For any tuple $a$ and any set $A$, the type $\stp(a/A_0)$ is foreign to $\P_0$, where $A_0 = \dcl(A,a) \cap \cl_{\P_0}(A)$. Moreover, it has the same $\wU$-rank as $\tp(a/A)$.
\end{lemma}
\begin{proof}
Suppose that $\stp(a/A_0)$ is not foreign to the family of types of $\wU$-rank zero. Thus, there is some $a_0\in \dcl(A_0,a)\setminus \acl(A_0)$ such that $\tp(a_0/A_0)$ is internal to the family of types of $\wU$-rank zero. That is, there are is some $C\ind_{A_0} a$ and some $b_1,\ldots,b_n$ with $\wU(b_i/A_0C)=0$ such that $a_0 \in \dcl(A_0C,b_1,\ldots,b_n)$. Hence we have that $\wU(a_0/A_0) = 0$. On the other hand, notice that $a_0\in \dcl(A,a)$ by definition of $A_0$ and moreover that $\wU(A_0/A) = 0$ since any finite tuple of elements from $A_0$ has $\wU$-rank zero over $A$ again by Lascar inequalities. Thus 
$$
\wU(a_0/A) \le \wU(a_0A_0/A) \le \wU(a_0/A_0 )\oplus \wU(A_0/A) = 0
$$
and so $a_0\in A_0$, a contradiction. Finally, the second part of the statement follows once more by the Lascar inequalities since $\wU(A_0/A)=0$.
\end{proof}

\begin{defn}
We say that a complete type is {\em $\omega$-minimal} if every forking extension of it is also an $\omega$-forking extension. 
\end{defn}

\begin{lemma}\label{L:ExtMin}
A non-forking extension of an $\omega$-minimal type is again $\omega$-minimal.
\end{lemma}
\begin{proof} To see this, let $q$ be a non-forking extension of an $\omega$-minimal type $p$ with parameters over $A$. Assume that $q=p|B$ and consider a forking extension $q'$ of $q$ over a set $B'$. Let $a$ be a realization of $q'$. Notice that $a\nind_B B'$ and $a\ind_A B$. Thus $a\nind_A BB'$ and so $a\nwind_A BB'$ since $p=\tp(a/A)$ is $\omega$-minimal. Moreover, we obtain that $a\nwind_B B'$ by transitivity since $a\wind_A B$, yielding that $q'$ is an $\omega$-forking extension of $q=\tp(a/B)$. Thus, the type $q$ is also $\omega$-minimal.
\end{proof}

\begin{remark}\label{R:Reg} If an $\omega$-minimal  type $p$ has ordinal $\wU$-rank, then every forking extension of it has strictly smaller $\wU$-rank. Hence, using the Lascar inequalities it is easy to see that any $\omega$-minimal stationary type of monomial $\wU$-rank is regular. Namely, if $p$ is an $\omega$-minimal stationary type with $\wU(p)=\omega^\alpha$ but there is a forking extension $p'$ of $p$ which is non-orthogonal to $p$, then there is set $A$ and realizations $a$ of $p|A$ and  $a'$ of $p'|A$ with $a\nind_A a'$. However, this implies that $\wU(a/Aa')<\wU(a/A)=\omega^\alpha$ since $\tp(a/A) = p|A$ is $\omega$-minimal and also that $\wU(a'/A) = \wU(p') < \wU(p) = \omega^\alpha$, yielding that 
$$
\omega^\alpha = \wU(a/A) \le \wU(a/Aa') \oplus \wU(a'/A) < \omega^\alpha, 
$$
a contradiction. 
\end{remark}

\begin{prop}\label{P:Min-For}
A stationary type is $\omega$-minimal if and only if it is foreign to $\P_0$.
\end{prop}
\begin{proof}
Assume first that $p$ is $\omega$-minimal but it is not foreign to the family of type of $\wU$-rank zero. Thus, there is some set $A$, some realization $a$ of $p|A$ and some tuple $\bar b=(\bar b_1,\ldots,b_n)$ with each $\tp(b_i/A)$ of $\wU$-rank zero such that $a\nind_A \bar b$. As $p$ is $\omega$-minimal, then so is $\tp(a/A)$ and so $a\nwind_A \bar b$. It then follows that $\bar b\nwind_A a$ by symmetry and so $\wU(\bar b/A) >0$, a contradiction.

For the other direction, suppose  towards a contradiction that $p=\tp(a/A)$ is foreign to $\P_0$ but there is some tuple $b$ such that $a\nind_A b$ and $a\wind_A b$. We then have that  $\cb(\stp(b/Aa))$ is not algebraic over $A$ and so $\cb(\stp(b/Aa))\nind_{A} a$. Consequently, there is some finite tuple $c\in \cb(\stp(b/Aa))$ such that $a\nind_{A} c$ and so $\wU(c/A) > 0$, since $p=\tp(a/A)$ is foreign to $\P_0$. On the other hand, as $b\wind_A a$, Proposition \ref{P:CanBase} yields that $\wU(c/A) = 0$, a contradiction. Therefore the type $\tp(a/A)$ is an $\omega$-minimal extension of $p$.
\end{proof}

As a consequence of Lemma \ref{L:Decomp} and Proposition \ref{P:Min-For} we obtain the following:
\begin{cor}\label{C:Ext}
Any stationary type $p=\tp(a/A)$ has an $\omega$-minimal extension of the same $\wU$-rank, namely the type $\tp(a/\dcl(Aa)\cap \cl_{\P_0}(A))$.
\end{cor}

The next result shows the existence of many regular types in a flat theory.

\begin{theorem}\label{T:Main2}
If the type $p$ has rank $\wU(p)=\beta + \omega^\alpha n$, with $n>0$ and $\beta \ge\omega^{\alpha+1}$ or $\beta=0$, then it has a non-$\omega$-forking extension $q$ which is not weakly orthogonal to an $\omega$-minimal regular type of $\wU$-rank $\omega^\alpha$.
\end{theorem}
\begin{proof} Let $p=\tp(a/A)$ and suppose that $\wU(a/A)=\beta + \omega^\alpha n$ with $n>0$ and $\beta \ge\omega^{\alpha+1}$ or $\beta=0$. Let $b$ be a tuple such that $\wU(a/Ab)= \beta + \omega^\alpha (n-1)$ and set $b'$ to be $\cb(\stp(a/Ab))$. Since $a\ind_{b'} Ab$, we have that $a\wind_{Ab'} Ab$ and so $\wU(a/Ab')=\wU(a/Ab)$. Thus, we may assume that $b'=b$.

The Lascar inequalities yield that $\wU(b/A) \ge\omega^\alpha$, and so we can find some set $B$ with $\wU(b/B)=\omega^\alpha$. Moreover, note that we may take $B$ containing $A$ in a way that $B\ind_{Ab} a$ and $B=\acl(B)$. Now, by Corollary \ref{C:Ext}, we know that $\tp(b/B_0)$ is $\omega$-minimal and has $\wU$-rank $\omega^\alpha$, where $B_0=\dcl(Bb)\cap\cl_{\P_0}(B)$. Furthermore, since $B_0\subseteq \dcl(Bb)$ we have that $B_0\ind_{Ab} a$ and so 
$$
\wU(a/B_0,b) = \wU(a/A,b) = \beta + \omega^\alpha (n-1)
$$ and 
$$
b=\cb(\stp(a/Ab))=\cb(\stp(a/B_0b)).
$$ 
Observe that since $\omega^\alpha = \wU(b/B_0)$, the type $\tp(b/B_0)$ cannot be algebraic and so $a\nind_{B_0} b$. As  $\tp(b/B_0)$ is $\omega$-minimal we then have $
\wU(b/B_0,a) < \wU(b/B_0)=\omega^\alpha
$ and hence $\omega^\alpha = \wU(b/B_0) \ge \wU(b/B_0a) + \omega^\alpha$. Whence 
$$
\wU(a/B_0) \ge \wU(a/B_0b) +\omega^\alpha = \beta + \omega^\alpha (n-1) +\omega^\alpha = \wU(a/A)
$$
by the Lascar inequalities and so $a\wind_A B_0$. Since $\tp(b/B_0)$ is $\omega$-minimal of monomial $\wU$-rank, it is regular by Remark \ref{R:Reg}. This finishes the proof.
\end{proof}

\begin{cor}\label{C:Reg}
If $p$ is foreign to $\P_0$ and $\wU(p)=\beta + \omega^\alpha n$, with $n>0$ and $\beta \ge\omega^{\alpha+1}$ or $\beta=0$, then $p$ is non-orthogonal to an $\omega$-minimal regular type of $\wU$-rank $\omega^\alpha$. 
\end{cor}
\begin{proof}
By Theorem \ref{T:Main2} there exists a non-$\omega$-forking extension $q$ of $p$ which is not weakly orthogonal to an $\omega$-minimal regular type $q'$ of $\wU$-rank $\omega^\alpha$. Since $p$ is also $\omega$-minimal by Proposition \ref{P:Min-For}, the type $q$ is indeed a non-forking extension of $p$ and so $p$ is not orthogonal to $q'$.
\end{proof}



\begin{remark} So far all results given in this section follow from the fact that the $\omega$-forking independence is an independence relation (in the sense of Theorem \ref{T:Main1}) with a well-behaved notion of rank. In other words, if in a stable theory we have an independence relation $\ind^*$ then one can define the corresponding notions of $\U_*$-rank, $*$-minimality, $*$-flatness and all results of this section adapt to this context.
\end{remark}


\subsection{Hereditarily triviality} In this subsection we awill show that types which are not foreign to $\P_0$ must have finite weight. For this, we introduce the following notion.

Let $\lambda$ denote an arbitrary cardinal. 

\begin{defn}
A partial type $\pi$ over $A$ is {\em hereditarily $\lambda$-trivial} if for any $a$ realizing $\pi$, any set $B\supseteq A$ and any independent sequence $I$ over $B$, there is some  $J\subseteq I$ with $|J|<\lambda$ for which  $aJ \ind_B I\setminus J$. 
\end{defn}

Observe that any hereditarily $\lambda$-trivial complete type has weight strictly smaller than $\lambda$. However, in Exercise 3.17 \cite[Chapter V]{SheClas} it is given an example of a finite weight type $p$ which is not hereditarily $w(p)$-trivial.  

\begin{expl}\label{E:1}
Consider an infinite vector space over a finite field, and let $I$ be a linearly independent set. Fix a finite set $J\subseteq I$ with $|J|>1$ and let $a=\sum_{x\in J} x$. Then there is no finite subset $J'$ of $I$ with $|J'| \le w(a) = 1$ such that $I\setminus J'$ is independent from $J' a$.  
\end{expl}

Now, we show some basic lemmata on hereditarily trivial types.

\begin{lemma}\label{L:Her1}
Assume $a\ind_A B$ with $A\subseteq B$. If $\tp(a/B)$ is hereditarily $\lambda$-trivial, then so is $\tp(a/A)$.
\end{lemma}
\begin{proof} Let $I$ be an independent sequence over $C\supseteq A$, and consider a set $B'$ such that $B'\equiv_{Aa} B$ and $B'\ind_{Aa} CI$. Thus $B'\ind_A CI$ by transitivity and invariance, and so the sequence $I$ is independent over $C\cup B'$. As $\tp(a/B)$ is hereditarily $\lambda$-trivial, so is $\tp(a/B')$ and hence there exists some subset $J$ of $I$ with $|J|<\lambda$ such that $I\setminus J\ind_{B'C} a J$. On the other hand, as $B'\ind_A CI$ we have that $B'\ind_C I$ and so the sequence $I\setminus J$ is independent from $aJ$ over $C$ by transitivity, as desired. \end{proof} 

\begin{lemma}\label{L:Her2}
 If $\tp(a/A)$ is hereditarily $\lambda_1$-trivial, and $\tp( b/A, a)$ is hereditarly $\lambda_2$-trivial, then $\tp( a b/A)$ is hereditarily $(\lambda_1+\lambda_2)$-trivial. 
\end{lemma}
\begin{proof} Consider an independent sequence $I$ over a set $B\supseteq A$. As $\tp(a/A)$ is hereditarily $\lambda_1$-trivial, there exists some $J_1\subseteq I$ with $|J_1|<\lambda_1$ and $J_1 a\ind_B I\setminus J_1$. Thus, the sequence $I\setminus J_1$ is independent over $B aJ_1$. Since $\tp(b/A,a)$ is hereditarily $\lambda_2$-trivial, we can find a subset $J_2\subseteq I\setminus J_1$ such that $J_2 b\ind_{BaJ_1} I\setminus (J_1\cup J_2)$ and $|J_2|<\lambda_2$. Therefore, we get $J_1J_2 a b\ind_B I\setminus (J_1\cup J_2)$ by transitivity. Hence, the type $\tp(a b/A)$ is hereditarily $(\lambda_1+\lambda_2)$-trivial, as desired.  \end{proof}

As a consequence we obtain:

\begin{prop}\label{P:AnHT}
Suppose that $\lambda$ is infinite. A finitary type analysable in the family of hereditarily $\lambda$-trivial types is itself hereditarily $\lambda$-trivial.
\end{prop}
\begin{proof}
Firstly we show that a finitary type $p=\tp(a/A)$ that is internal to a family of hereditarily $\lambda$-trivial types is itself hereditarily $\lambda$-trivial. To do so, suppose that there is some set $B$ with $a\ind_A B$ and some tuple $\bar b=(b_1,\ldots,b_n)$ with $\tp(b_i/B)$ hereditarily $\lambda$-trivial such that $a\in \dcl(B,\bar b)$. It follows from the definition that $\tp(b_i/B,b_{<i})$ is also hereditarily $\lambda$-trivial and so is $\tp(\bar b/B)$ by Lemma \ref{L:Her2}. Again it follows easily from the definition that $\tp(a/B)$ is also hereditarily $\lambda$-trivial, and then so is $\tp(a/A)$ by Lemma \ref{L:Her1}. 

Now, suppose that $p=\tp(a/A)$ is analysable in a family of hereditarily $\lambda$-trivial types. By definition, there is a sequence $(a_i)_{i<\alpha}$ in $\dcl(A,a)$ such that each type $\tp(a_i/A,(a_j)_{j<i})$ is internal to the given family of hereditarily $\lambda$-trivial types, and $a\in\acl(A,(a_i)_{i<\alpha})$. We have just seen in the paragraph above that each $\tp(a_i/A,(a_j)_{j<i})$ is hereditarily $\lambda$-trivial. Hence, as $a$ is a finite tuple, we have that $\alpha$ is indeed a natural number and so applying $\alpha$ many times Lemma \ref{L:Her2} we obtain that $\tp((a_i)_{i<\alpha}/A)$ is also  hereditarily $\lambda$-trivial. Whence, the type $\tp(a/A)$ is hereditarily $\lambda$-trivial as well since $a\in\acl(A,(a_i)_{i<\alpha})$.
\end{proof}

For an infinite cardinal $\lambda$, let $\P_{\rm ht,\lambda}$ be the family of all hereditarily $\lambda$-trivial types. It follows from the result above that that given a set $A$, the set $\cl_{\P_{\mathrm{ht},\lambda}}(A)$ of all tuples $b$ such that $\tp(b/A)$ is hereditarily $\lambda$-trivial is a closure operator. Alternatively, this can be easily seen using Lemma \ref{L:Her1} and \ref{L:Her2}.

Similarly as in Lemma \ref{L:Decomp} (or by \cite[Corollary 6]{Udi1}) we get the existence of types foreign to $\P_{\mathrm{ht},\lambda}$.

\begin{cor}
For any tuple $a$ and any set $A$, the type $\tp(a/A_0)$ is foreign to the family of hereditarily $\lambda$-trivial types, where $A_0=\dcl(A,a)\cap \cl_{\P_{\mathrm{ht},\lambda}}(A)$. 
\end{cor}

Now, we focus our attention to the family of hereditarily $\omega$-trivial types, which contains $\P_0$ as it is shown in the next lemma.

\begin{lemma}\label{L:Rank0HT}
Let $\bar a$ be a possibly infinite tuple  and $a$ a finite tuple such that $\bar a$ is contained in $\acl(a)$. If a type $p=\tp(\bar a/A)$ has $\wU$-rank zero, then it is hereditarily $k$-trivial for some natural number $k$.
\end{lemma}
\begin{proof}
Suppose that $p=\tp(\bar a/A)$ has $\wU$-rank $0$, where $\bar a\subseteq\acl(a)$ and $a$ is a finite tuple. Let $B$ and $I$ be as in the definition of hereditarily $\omega$-trivial and consider an $(|A|+|T|)^+$-saturated model $M$ containing $B,I$ and $a$. Since 
$$\wU(\bar a/M) \le \wU( \bar a/A )=0
$$ 
we have that $\bar a\wind_A M$. Now, let $a'$ be a finite tuple such that $a'\equiv_{A\bar a} a$ with $a'\ind_{A\bar a} M$ and so $a'\wind_A M$ by transitivity. It then follows that $\tp(a'/M)$ does not $k$-fork over $B$ for some natural number $k$ by definition. Hence, Proposition \ref{P:Equiv1} yields the existence of a subset $J$ of $I$ with $|J|<k$ such that $I\setminus J$ is independent from $Ja'$ over $B$. Whence, as $\bar a\subseteq \acl(a')$ by invariance, we get that  $I\setminus J$ is independent from $J\bar a$ over $B$ and so $p=\tp(\bar a/A)$ is hereditarily $k$-trivial.
\end{proof}

\begin{prop}\label{P:NotForeign}
If a type is not foreign to $\P_{\mathrm{ht},\omega}$, then it dominates an hereditarily $\omega$-trivial type and it is non orthogonal to a type of weight one.
\end{prop}
\begin{proof}
Let $p=\tp(a/A)$ be a type which is not foreign to $\P_{\mathrm{ht},\omega}$. Thus, there exists some element $a_0\in \dcl(A,a)\setminus \acl(A)$ such that $\tp(a_0/A)$ is internal to $\P_{\mathrm{ht},\omega}$ by Fact \ref{F:Int}. Hence, by Proposition \ref{P:AnHT} the latter type is indeed hereditarily $\omega$-trivial and clearly it is dominated by $\tp(a/A)$. 

Now, as $\tp(a_0/A)$ has finite weight it is non-orthogonal to a type of weight one by a result of Hyttinen, see \cite[Proposition 5.6.6]{Bue}. Hence, the type $\tp(a/A)$ is also non-orthogonal to a type of weight one.
\end{proof}

\subsection{Flatness and finite weight} Now, we are ready to prove that a flat theory is strong, {\it i.e.} every type has finite weight. In fact, we obtain a local version of this.

\begin{theorem}\label{T:Strong}
A finitary type $p$ with $\wU(p)<\infty$ has finite weight and therefore it is non-orthogonal to a type of weight one.
\end{theorem}

\begin{proof}
We proceed by induction on the $\wU$-rank of the type. The case of $\wU$-rank $0$ follows by Lemma \ref{L:Rank0HT} and the fact that an hereditarily trivial type has finite weight. 

Now, let $p\in S(\emptyset)$ be a finitary type and assume that $\wU(p) = \beta +\omega^\alpha \cdot n$ with $n>0$ and $\beta \ge \omega^{\alpha+1}$ or $\beta=0$. Let $a$ be a realization of $p$ and set $A = \dcl(a)\cap {\rm cl}_{\mathbb P_0}(\emptyset)$. By Lemma \ref{L:Decomp} the type $\stp(a/A)$ is foreign to $\mathbb P_0$ and $\wU(a/A) = \wU(p)$. Thus, applying Corollary \ref{C:Reg} we can find an $\omega$-minimal  regular type $q$ which is non-orthogonal to $\tp(a/A)$. Let $C$ and $b$ be such that $a\ind_A C$ and $b\models q|C$ with $a\nind_C b$. Note that the latter implies the existence of some imaginary element 
$$
a_0\in \dcl(\cb(b/C,a))\setminus \acl(C).
$$
Thus  $a_0\in\acl(C,a)$ and also $a_0\in\dcl(b_0,\ldots,b_m)$ for some initial segment $b_0,\ldots,b_m$ of a Morley sequence in $\stp(b/C,a)$. Hence, we then have that 
$$
w(a_0/C) \le w(b_0,\ldots,b_m/C ) \le m, 
$$
since $\tp(b/C) = q|C$ is regular and so of weight $1$. 

Since $a\ind_A C$, the type $\tp(a/C)$ is also foreign to $\mathbb P_0$ and thus $\wU(a_0/C)>0$, as $a\nind_C a_0$ by the choice of $a_0$. Hence, by the Lascar inequality 
$$
\wU(a/C,a_0) + \wU(a_0/C ) \le \wU(a,a_0/C) = \wU(a/C)
$$
we then have that $\wU(a/C,a_0)<\wU(a/C)$. Therefore, putting al together we get
$$
w(a/A) = w(a/C) = w(a,a_0/C) \le w(a/C,a_0) + w(a_0/C) <\omega,
$$
since by induction the type $\tp(a/C,a_0)$ has finite weight. 

Finally, notice that $\tp(A)$ has $\wU$-rank zero, since $A\subseteq \cl_{\P_0}(\emptyset)$. As $a$ is finite and $A\subseteq\dcl(a)$, using Lemma \ref{L:Rank0HT} we then see that $\tp(A)$ has finite weight, which yields that $\tp(a)$ has also finite weight, since 
$$
w(a) = w(a,A) \le w(a/A) + w(A).
$$
This finishes the first part of the statement. For the second, it suffices to notice as before that a type of finite weight is non-orthogonal to a type of weight one.
\end{proof}

As an immediate consequence we obtain:

\begin{cor}
Any flat theory is strong.
\end{cor}

In the light of these results, it seems reasonable to ask the following:

\begin{quest}
In a flat theory, is every type non-orthogonal to a regular type? Or, is every type hereditarily $1$-trivial type non-orthogonal to a regular type? 
\end{quest}

\section{Flat groups}
In this final section we describe the structure of type-definable groups in flat theories. It turns out that this resemblances to the structure of a superstable group, since in the framework of groups we can find enough regular types.  More precisely, we  recover \cite[Corollary 5.3]{Udi2} where a superstable group $G$ is shown to admit a normal series of definable subgroups
$$
G=G_0\unrhd G_1 \unrhd \dots \unrhd G_m \unrhd \{1\}
$$
such that each group $G_i/G_{i+1}$ is $p_i$-semi-regular for some regular type $p_i$. We refer the reader to \cite[Chapter 7]{Pil} for the general theory of $p$-simplicity and semi-regularity; which was originally introduced (in a different way) in \cite[Chapter V]{SheClas}. We recall some of the basic definitions.

Fix a regular type $p$. Recall that a stationary type $q$ is said to be {\em hereditarily orthogonal} to $p$ if $p$ is orthogonal to any extension of $q$. A stationary type $q$ is {\em $p$-simple} if for some set $B$ with $p$ and $q$ based on $B$, there exist $c\models q|B$ and an independent sequence $I$ of realizations of $p|B$ such that $\stp(c/B,I)$ is hereditarily orthogonal to $p$. The type $q$ is $p$-semi-regular if it is $p$-simple and domination-equivalent to $p^{(n)}$. In fact, a $p$-simple type $q=\stp(a/A)$ is $p$-semi-regular if and only if $\tp(d/A)$ is not hereditarily orthogonal to $p$ for every $d\in\dcl(A,a)\setminus \acl(A)$, see \cite[Lemma 7.1.18]{Pil} for a proof. Finally, concerning groups, we say that a group is {\em $p$-simple} if some (any) generic type is $p$-simple, and it is {\em $p$-semi-regular} group if some (any) generic is $p$-semi-regular.

The main facts concerning $p$-simplicity for groups, which we recall bellow, are shown by Hrushovski in \cite{Udi2}, see also \cite[Lemma 7.4.7]{Pil} for a proof.

\begin{fact}
Let $G$ be a type-definable group and let $q\in S_G(\emptyset)$ be some generic type.
\begin{enumerate}
\item If $q$ is not foreign to an $\emptyset$-invariant family $\mathbb P$ of types, then there exists a relatively definable  normal subgroup $N$ of $G$ of infinite index such that $G/N$ is $\mathbb P$-internal.
\item If $q$ is non-orthogonal to a regular type $p$, then there exists a relatively definable normal subgroup $N$ of $G$ such that $G/N$ is $p$-simple (even $p$-internal), and that a generic type of $G/N$ is non-orthogonal to $p$.
\end{enumerate}
\end{fact}

Before proceeding to analyse flat groups, we first see that the $\wU$-rank behaves as the $\U$-rank for groups. Note that a generic type $p\in S_G(A)$ has maximal $\wU$-rank: If $\tp(h/A)$ is another type, then taking $g\models p|A,h$  we get
\begin{align*}
\wU(h/A) & = \wU(h/A,g) = \wU (gh/A,g) \le \wU(gh/A) \\ &= \wU(gh/A,h) = \wU(g/A,h)  = \wU(p),
\end{align*}
since $g\wind_A h$, $h\wind_A g$ and $gh\wind A,h$. We then set $\wU(G)$ to be the $\wU$-rank of some (any) generic type; note that {\it a priori} a type of maximal $\wU$-rank might not be generic. Similarly, we can define the $\wU$-rank of a coset space to be the $\wU$-rank of its generic type. More precisely, if $g$ is generic of $G$ over $A$ and $E(x,y)$ is the equivalence relation $x^{-1}y\in H$ for some relatively definable subgroup $H$ of $G$, then $\tp(g_E/A)$ is the generic for $G/H$ and moreover note that $\tp(g/A,g_E)$ is a generic for the coset $gH$. Thus, since $\wU(gH)=\wU(H)$, using the Lascar inequalities, we get
$$
\wU(H) + \wU(G/H) \le \wU(G) \le \wU(H) \oplus \wU(G/H).
$$

The following key fact is a generalization of Example \ref{E:1}.

\begin{lemma}\label{L:GroupHT}
A generic type of an infinite type-definable group is not hereditarily $k$-trivial for any natural number $k$. In particular, there is no hereditarily $k$-trivial partial type defining an infinite group. 
\end{lemma}
\begin{proof} Let $G$ be an infinite type-definable group and suppose, towards a contradiction, that the principal generic $p\in S_G(\emptyset)$  of $G$ is hereditarily $k$-trivial. Now, let $(a_i)_{i<k+2}$ be an independent sequence of realizations of $p$ and set $a=\prod_{i<k+1} a_i$. As $a$ realizes $p$, by assumption there is some subset $J$ with $|J|<k$ such that $(a_i)_{i\in J} a \ind (a_i)_{i \not\in J}$. Thus, the definition of $a$ yields the existence of some $k\not\in J$ such that $a_k \in \dcl(a,(a_i)_{i\in J})$ and so $a_k$ is independent from itself. This implies that $p$ is an algebraic type and so $G$ is finite, a contradiction. 

The second part of the statement follows from the fact that if $G$ is type-defined by an hereditarily $k$-trivial type, then so is any generic.
\end{proof}
As a consequence, we then have by Lemma \ref{L:Rank0HT} that a group of $\wU$-rank zero must be finite. More generally, we obtain:
\begin{lemma}\label{L:GroupRank}
Let $G$ be a type-definable group and let $H$ be a relatively definable subgroup of $H$. Then
$\wU(G)=\wU(H)<\infty$ if and only if $G/H$ is finite.
\end{lemma}
\begin{proof}
It is enough to use the above Lascar inequalities for groups and notice that by Lemma \ref{L:Rank0HT} a type-definable group has $\wU$-rank $0$ if and only if it is finite.
\end{proof}

\begin{cor}\label{C:DCC}
Let $G$ be a type-definable group with $\wU(G)<\infty$. Then, there is no infinite sequence of relatively definable subgroups, each having infinite index in its predecessor.
\end{cor}

Now, we can obtain the semi-regular decomposition for flat groups.

\begin{theorem}\label{T:Group}
Let $G$ be a type-definable group with $\wU(G)<\infty$. Then, there exist finitely many regular types $p_0,\ldots,p_m$ and a series of relatively definable subgroups
$$
G=G_0 \unrhd G_1 \unrhd \dots \unrhd G_m \unrhd \{1\}
$$
such that each group $G_i/G_{i+1}$ is $p_i$-internal and $p_i$-semi-regular.
\end{theorem}
\begin{proof}
We proceed by induction on the $\wU$-rank. Assume that $\wU(G) = \beta +\omega^\alpha \cdot n$ for some ordinal $\beta \ge \omega^{\alpha+1}$ or $\beta =0$ and some $n\ge 0$. Moreover, since a group of $\wU$-rank zero is finite we may assume that $n>0$. 

We first claim that a generic type of $G$ is foreign to the set of types of $\wU$-rank strictly smaller than $\omega^\alpha$.  Otherwise, the previous fact yields the existence of a relatively definable normal subgroup $N$ of $G$ of infinite index such that $G/N$ is internal to the family of types of $\wU$-rank strictly smaller than $\omega^\alpha$. Thus, we then have that $\wU(G/N)<\omega^\alpha$ by the Lascar inequalities and so, the inequation 
$$
\beta+\omega^\alpha \cdot n = \wU(G)\le \wU(N) \oplus \wU(G/N)
$$
yields that $\wU(N)=\wU(G)$, a contradiction by Lemma \ref{L:GroupRank} since $G/N$ is infinite. Therefore, any generic type of  $G$ is foreign, and so orthogonal, to any type of $\wU$-rank strictly smaller than $\omega^\alpha$. In particular, it is foreing to $\mathbb P_0$ and consequently, by Corollary \ref{C:Reg}, some generic type $q$ of $G$ is non-orthogonal to an $\omega$-minimal regular type $p$ of $\wU$-rank $\omega^\alpha$. Thus, the second point of the previous fact yields the existence of a relatively definable normal subgroup $H$ such that $G/H$ is $p$-internal (so $p$-simple), and some generic type $q'$ of $G/H$ is non-orthogonal to $p$. 

Assume $p,q$ and $q'$ are stationary over $A$.  Since $p$ is $\omega$-minimal, any forking extension of $p$ has smaller $\wU$-rank and hence, the first part of the proof implies that $q$ is orthogonal to any forking extension of $p$. Whence, the same is true of $q'$ since $q$ dominates $q'$. Consequently, a standard argument (see the proof of \cite[Corollary 7.1.19]{Pil}) yields that $q'$ is $p$-semi-regular. Namely, as $q'$ is $p$-internal, there is some set $B$ containing $A$, some $c_1,\ldots,c_k$ realizing $p$ and some $a\models q'|B$ such that $a\in \dcl(B,c_1,\ldots,c_k)$. Fix some $d\in\dcl(a,A)\setminus \acl(A)$, and note then that there exists some $m\le k$ such that $d\ind_A B c_{<m}$ but $d\nind_{B,c_{<m}} c_{m}$. Setting $r=\tp(c_m/B,c_{<m})$, we clearly have that $\tp(d/A)$ is non-orthogonal to $r$ and then so is $q'=\tp(a/A)$, since $d\in\dcl(A,a)$. Thus, necessarily $r$ must be a non-forking extension of $p$, as $q'$ is orthogonal to any forking extension. This implies that $r$ is a regular type, as so is $p$, and moreover that $p$ and $r$ are non-orthogonal. Hence, we then have $\tp(d/A)$ is non-orthogonal to $p$, yielding that $q'$ is semi-regular by \cite[Lemma 7.1.18]{Pil}, say. Therefore, we have shown that $G/H$ is $p$-internal and $p$-semi-regular. 

Finally, as $G/H$ is infinite and so $\wU(H)<\wU(G)$ by Lemma \ref{L:GroupRank}, the inductive hypothesis applied to $H$ yields the statement.
\end{proof}

To finish the paper, a question:

\begin{quest}
Is there a flat non-superstable group?
\end{quest}

\end{document}